\documentclass[number,3p]{elsarticle}

\usepackage{amsmath,amsthm,amssymb}

\theoremstyle{plain}
\newtheorem{teorema}{Theorem}
\newtheorem*{corolario}{Corollary}
\newtheorem{lema}{Lemma}

\theoremstyle{definition}
\newtheorem*{remark}{Remark}

\newcommand{\C}{\mathbb{C}}
\newcommand{\R}{\mathbb{R}}
\newcommand{\Z}{\mathbb{Z}}
\newcommand{\N}{\mathbb{N}}
\newcommand{\Real}{\operatorname{Re}}
\newcommand{\Imag}{\operatorname{Im}}

\newcommand{\Res}{\operatorname{Res}}

\allowdisplaybreaks[3] 

\journal{J.\ Math.\ Math.\ Anal.\ Appl.}

\makeatletter
\def\ps@pprintTitle{%
     \let\@oddhead\@empty
     \let\@evenhead\@empty
     \def\@oddfoot{\footnotesize
     This paper has been published in
     \textit{J.\ Math.\ Anal.\ Appl.} \textbf{372} (2010), 470--485;
     doi:10.1016/j.jmaa.2010.07.029%
     \hfill}%
     \let\@evenfoot\@oddfoot}
\makeatother

\begin{document}

\title{Mean convergence of Fourier-Dunkl series}

\author[rioja]{\'Oscar~Ciaurri\fnref{asesora}}
\address[rioja]{CIME and Departamento de Matem\'aticas y Computaci\'on,
Universidad de La Rioja, 
26004 Logro\~no, Spain}
\ead{oscar.ciaurri@unirioja.es}
\fntext[asesora]{Supported by grant MTM2009-12740-C03-03, Ministerio de Ciencia e Innovaci\'on, Spain}

\author[mario]{Mario~P\'{e}rez\corref{cor1}\fnref{asesora,dga}}
\address[mario]{IUMA and Departamento de Matem\'aticas,
Universidad de Za\-ra\-go\-za,
50009 Za\-ra\-go\-za, Spain}
\ead{mperez@unizar.es}
\cortext[cor1]{Corresponding author}
\fntext[dga]{Supported by grant E-64, Gobierno de Arag\'on, Spain}

\author[reyes]{Juan~Manuel~Reyes}
\address[reyes]{Departament de Tecnologia, 
Universitat Pompeu Fabra,
08003 Bar\-ce\-lo\-na, Spain}
\ead{reyes.juanmanuel@gmail.com}

\author[rioja]{Juan~Luis~Varona\fnref{asesora}}
\ead{jvarona@unirioja.es}

\begin{abstract}
In the context of the Dunkl transform a complete orthogonal system arises in a
very natural way. This paper studies the weighted norm convergence of the 
Fourier series expansion associated to this system. We establish conditions on 
the weights, in terms of the $A_p$ classes of Muckenhoupt, which ensure the 
convergence. Necessary conditions are also proved, which for a wide class of 
weights coincide with the sufficient conditions.
\end{abstract}


\maketitle

\section{Introduction}

For $\alpha>-1$, let $J_{\alpha}$ denote the Bessel function of order
$\alpha$:
\[
   J_{\alpha}(x)
   = \left( \frac{x}{2} \right)^{\alpha}
   \sum_{n=0}^{\infty} \frac{(-1)^n (x/2)^{2n}}{n! \, \Gamma(\alpha+n+1)}
\]
(a classical reference on Bessel functions is~\cite{Wat}). 
Throughout this paper, by $\frac{J_{\alpha}(z)}{z^\alpha}$ we denote the even 
function
\begin{equation}
\label{serie}
   \frac{1}{2^\alpha}
   \sum_{n=0}^{\infty} \frac{(-1)^n (z/2)^{2n}}{n! \, \Gamma(\alpha+n+1)},
   \quad z \in \C.
\end{equation}
In this way, for complex values of~$z$, let
\[
   \mathcal{I}_{\alpha}(z)
   = 2^{\alpha} \Gamma(\alpha+1) \frac{J_{\alpha}(iz)}{(iz)^{\alpha}}
   = \Gamma(\alpha+1) \sum_{n=0}^{\infty}
   \frac{(z/2)^{2n}}{n!\,\Gamma(n+\alpha+1)};
\]
the function $\mathcal{I}_{\alpha}$ is a small variation of the so-called 
modified Bessel function of the first kind and order $\alpha$, usually denoted 
by~$I_{\alpha}$. Also, let us take
\[
   E_{\alpha}(z)
   =\mathcal{I}_{\alpha}(z) + \frac{z}{2(\alpha+1)}\, \mathcal{I}_{\alpha+1}(z),
   \quad z \in \C.
\]
These functions are related with the so-called Dunkl transform on the real line 
(see~\cite{Dunkl2} and~\cite{Jeu} for details), which is a generalization of the 
Fourier transform. In particular, $E_{-1/2}(x)=e^{x}$ and the Dunkl transform of 
order $\alpha=-1/2$ becomes the Fourier transform.
Very recently, many authors have been investigating the behaviour of the
Dunkl transform with respect to several problems already studied for the Fourier
transform; for instance, Paley-Wiener theorems~\cite{AndJeu}, 
multipliers~\cite{BeCiVa}, uncertainty~\cite{RoVo}, Cowling-Price's theorem~\cite{MeTri}, 
transplantation~\cite{NoSt}, Riesz transforms~\cite{NoSt2}, and so on.
The aim of this paper is to pose and analyse in this new context the weighted 
$L^p$ convergence of the associated Fourier series in the spirit of the 
classical scheme which, for the trigonometric Fourier series, can be seen in 
Hunt, Muckenhoupt and Wheeden's paper~\cite{HMW}.

The function $\mathcal{I}_{\alpha}$ is
even, and $E_{\alpha}(ix)$ can be expressed as
\[
   E_{\alpha}(ix)
   = 2^{\alpha} \Gamma(\alpha+1) \left( \frac{J_{\alpha}(x)}{x^{\alpha}}
   + \frac{J_{\alpha+1}(x)}{x^{\alpha+1}} xi \right).
\]
Let $\{s_j\}_{j\geq 1}$ be the increasing sequence of positive zeros of 
$J_{\alpha+1}$. The real-valued function $\Imag E_{\alpha}(ix) =
\frac{x}{2(\alpha+1)}\,\mathcal{I}_{\alpha+1}(ix)$ is odd and its zeros are
$\{s_j\}_{j\in \Z}$ where $s_{-j}=- s_j$ and $s_0=0$. 
In connection with the Dunkl transform on the real line, two of the authors introduced 
the functions~$e_j$, $j \in \Z$, as follows:
\begin{align*}
   e_0(x) &= 2^{(\alpha+1)/2} \Gamma(\alpha+2)^{1/2},
   \\
   e_j(x)
   &= \frac{2^{\alpha/2}\Gamma(\alpha+1)^{1/2}}{|\mathcal{I}_{\alpha}(i s_j)|}\,
   E_{\alpha}(i s_j x),
   \quad j \in \Z \setminus \{0\}.
\end{align*}
The case $\alpha=-1/2$ corresponds to the classical trigonometric Fourier setting: 
$\mathcal{I}_{-1/2}(z) = \cos(iz)$, $\mathcal{I}_{1/2}(z) = \frac{\sin(iz)}{iz}$, 
$s_j = \pi j$, $E_{-1/2}(i s_j x) = e^{i\pi jx}$, and $\{e_j\}_{j \in \Z}$ is 
the trigonometric system with the appropriate multiplicative constant so that 
it is orthonormal on $(-1,1)$ with respect to the normalized Lebesgue 
measure $(2\pi)^{-1/2}\,dx$.

For all values of $\alpha > -1$, in~\cite{CiVa} the sequence $\{e_j\}_{j\in\Z}$ 
was proved to be a complete orthonormal system in $L^2((-1, 1), d\mu_{\alpha})$, 
$d\mu_{\alpha}(x) = (2^{\alpha+1}\Gamma(\alpha+1))^{-1}|x|^{2\alpha+1}\,dx$.
That is to say
\[
   \int_{-1}^1 e_j(x)\overline{e_k(x)} \, d\mu_{\alpha}(x) = \delta_{jk} 
\]
and for each $f\in L^2((-1,1), d\mu_{\alpha})$ the series
\[
   \sum_{j=-\infty}^{\infty} \left(\int_{-1}^1 f(y) \overline{e_j(y)} \, d\mu_{\alpha}(y) \right)
   e_j(x),
\]
which we will refer to as Fourier-Dunkl series,
converges to $f$ in the norm of $L^2((-1,1), d\mu_{\alpha})$.
The next step is to ask for which $p\in (1,\infty)$, $p\neq 2$,
the convergence holds in $L^p((-1,1), d\mu_{\alpha})$. The problem is equivalent, by the
Banach-Steinhauss theorem, to the uniform boundedness on
$L^p((-1,1), d\mu_{\alpha})$ of the partial sum operators $S_n f$ given by
\[
   S_n f (x)=\int_{-1}^1 f(y ) K_n(x,y) \, d\mu_{\alpha}(y),
\]
where $K_n(x,y)=\displaystyle\sum_{j=-n}^n e_j(x)\overline{e_j(y)}$. 
We are interested in weighted norm estimates of the form
\[
   \| S_n(f) U \|_{L^p((-1,1), d\mu_{\alpha})}
   \leq C \| f V\|_{L^p((-1,1), d\mu_{\alpha})},
\]
where $C$ is a constant independent of $n$ and $f$, and $U$, $V$ are nonnegative 
functions on $(-1,1)$.

Before stating our results, let us fix some notation. The conjugate 
exponent of $p\in (1,\infty)$ is denoted by~$p'$.
That is,
\[
   \frac{1}{p} + \frac{1}{p'} = 1,
   \quad \text{or}
   \quad p' =\frac{p}{p-1}.
\]
For an interval $(a,b) \subseteq \R$, the Muckenhoupt class $A_p(a,b)$ consists 
of those pairs of nonnegative functions $(u,v)$ on $(a,b)$ such that
\[
   \left( \frac{1}{|I|} \int_{I}u(x) \, dx \right)
   \left(\frac{1}{|I|}\int_{I}v(x)^{-\frac{1}{p-1}} \, dx \right)^{p-1}\leq C ,
\]
for every interval $I\subseteq (a,b)$, with some constant $C > 0$ independent of~$I$. 
The smallest constant satisfying this property
is called the $A_p$ constant of the pair $(u,v)$.

We say that $(u,v)\in A_p^{\delta}(a,b)$ (where $\delta>1$) if $(u^{\delta} ,
v^{\delta})\in A_p(a,b)$. It follows from H\"{o}lder's inequality that
$A_p^{\delta}(a,b)\subseteq A_p(a,b)$.

If $u \equiv 0$ or $v \equiv \infty$, it is trivial that $(u,v)\in A_p(a,b)$ 
for any interval $(a,b)$.
Otherwise, for a bounded interval $(a,b)$, if $(u,v)\in A_p(a,b)$ then the 
functions~$u$ and $v^{-\frac{1}{p-1}}$ are integrable on $(a,b)$.

Throughout this paper, $C$ denotes a positive constant which may be different in each occurrence.

\section{Main results}

We state here some $A_p$ conditions which ensure the weighted $L^p$ boundedness
of these Fourier-Dunkl orthogonal expansions. For simplicity, we separate the
general result corresponding to arbitrary weights in two theorems, the first one
for $\alpha \geq -1/2$ and the second one for $-1<\alpha <-1/2$.

\begin{teorema}\label{tma1}
Let $\alpha \geq -1/2$ and $1<p<\infty$. Let $U$, $V$ be weights on
$(-1,1)$. Assume that
\begin{equation}
\label{Ap_tma1}
   \left(
   U(x)^p |x|^{(\alpha+\frac{1}{2})(2-p)},
   V(x)^p |x|^{(\alpha+\frac{1}{2})(2-p)}
   \right)
   \in A_p^{\delta}(-1,1)
\end{equation}
for some $\delta>1$ (or $\delta=1$ if $U=V$). Then there exists a constant $C$ 
independent of $n$ and~$f$ such that
\[
   \|S_n(f) U\|_{L^p((-1,1),d\mu_{\alpha})}
   \leq C \|f V\|_{L^p((-1,1), d\mu_{\alpha})}.
\]
\end{teorema}

\begin{teorema}\label{tma2}
Let $-1<\alpha <-1/2$ and $1<p<\infty$. Let $U$, $V$ be weights on $(-1,1)$. 
Let us suppose that $U$, $V$ satisfy the conditions
\begin{align}
\label{Ap1_tma2}
   \left(
   U(x)^p |x|^{(2\alpha+1)(1-p)},
   V(x)^p |x|^{(2\alpha+1)(1-p)}
   \right)
   &\in A_p^{\delta}(-1,1),
   \\
\label{Ap2_tma2}
   \left(
   U(x)^p |x|^{2\alpha +1},
   V(x)^p |x|^{2\alpha +1}
   \right)
   &\in A_p^{\delta}(-1,1)
\end{align}
for some $\delta>1$ (or $\delta=1$ if $U=V$). Then there exists a constant $C$ 
independent of $n$ and~$f$ such that
\[
   \|S_n(f) U\|_{L^p((-1,1),d\mu_{\alpha})}
   \leq C \|f V\|_{L^p((-1,1), d\mu_{\alpha})}.
\]
\end{teorema}

As we mentioned in the introduction, the case $\alpha = -1/2$ corresponds to 
the classical trigonometric case. Accordingly, \eqref{Ap_tma1} reduces then 
to $(U^p, V^p) \in A_p^\delta(-1,1)$. It should be noted also that taking real 
and imaginary parts in these Fourier-Dunkl series we would obtain the so-called 
Fourier-Bessel series on $(0,1)$ (see~\cite{Wng, BP1, BP2, GPRV2}), but the 
known results for Fourier-Bessel series do not give a proof of the above theorems. 
Also in connection with Fourier-Bessel series on $(0,1)$, Lemma~\ref{Apunif} below 
can be used to improve some results of~\cite{GPRV2}.

Theorems~\ref{tma1} and~\ref{tma2} establish some sufficient conditions for 
the $L^p$ boundedness. Our next result presents some necessary conditions. 
To avoid unnecessary subtleties, we exclude the trivial cases $U \equiv 0$ 
and $V \equiv \infty$.

\begin{teorema}
\label{tma3}
Let $-1 < \alpha$, $1 < p < \infty$, and $U$, $V$ weights on $(-1,1)$, 
neither $U \equiv 0$ nor $V \equiv \infty$. If there exists some constant 
$C$ such that, for every $n$ and every~$f$,
\[
   \|S_n(f) U\|_{L^p((-1,1),d\mu_{\alpha})}
   \leq C \|f V\|_{L^p((-1,1), d\mu_{\alpha})},
\]
then $U \leq C V$ almost everywhere on $(-1,1)$, and
\begin{align*}
   U(x)^p |x|^{(\alpha+\frac{1}{2})(2-p)} &\in L^1((-1,1),dx),
   \\
   \left(
   V(x)^p |x|^{(\alpha+\frac{1}{2})(2-p)} \right)^{-\frac{1}{p-1}}
   = V(x)^{-p'} |x|^{(\alpha+\frac{1}{2})(2-p')}
   &\in L^1((-1,1),dx),
   \\
   U(x)^p |x|^{2\alpha + 1} &\in L^1((-1,1),dx),
   \\
   \left(
   V(x)^p |x|^{(2\alpha + 1)(1-p)} \right)^{-\frac{1}{p-1}}
   = V(x)^{-p'} |x|^{2\alpha + 1}
   &\in L^1((-1,1),dx).
\end{align*}
\end{teorema}

Notice that the first two integrability conditions imply the other two if 
$\alpha \geq -1/2$, while the last two imply the other if $-1 < \alpha < -1/2$.

When $U$, $V$ are power-like weights, it is easy to check that the conditions 
of Theorem~\ref{tma3} are equivalent to the $A_p$ conditions~\eqref{Ap_tma1}, 
\eqref{Ap1_tma2},~\eqref{Ap2_tma2}. By power-like weights we mean finite 
products of the form $|x-t|^\gamma$, for some constants $t$,~$\gamma$. 
For these weights, therefore, Theorems~\ref{tma1}, \ref{tma2} and~\ref{tma3} 
characterize the boundedness of the Fourier-Dunkl expansions. 
For instance, we have the following particular case:

\begin{corolario}
Let $b, A, B \in \R$, $1 < p < \infty$, and
\[
   U(x) = |x|^b (1-x)^A (1+x)^B.
\]
Then, there exists some constant $C$ such that 
\[
   \|U S_nf \|_{L^p((-1,1),d\mu_{\alpha})}
   \leq C \| U f \|_{L^p((-1,1), d\mu_{\alpha})}
\]
for every $f$ and $n$
if and only if $-1 < Ap < p-1$, $-1 < Bp < p-1$ and
\[
   -1 + p \Bigl(\alpha + \frac{1}{2}\Bigr)_+ < b p + 2 \alpha + 1
   < p - 1 + p (2\alpha + 1) - p \Bigl(\alpha + \frac{1}{2}\Bigr)_+,
\]
where $(\alpha+\frac{1}{2})_+ = \max\{\alpha+\frac{1}{2}, 0\}$.
\end{corolario}

In the unweighted case ($U = V = 1$) the boundedness of the partial 
sum operators $S_n$, or in other words the convergence of the Fourier-Dunkl series, 
holds if and only if
\[
   \frac{4(\alpha+1)}{2\alpha+3} < p < \frac{4(\alpha+1)}{2\alpha+1}
\]
in the case $\alpha \geq -1/2$, and for the whole range $1 < p < \infty$ 
in the case $-1 < \alpha < -1/2$.

\begin{remark}
These conditions for the unweighted case are exactly the same as in the 
Fourier-Bessel case when the orthonormal functions are 
$2^{1/2} |J_{\alpha+1}(s_n)|^{-1} J_\alpha(s_n x) x^{-\alpha}$ 
and the orthogonality measure is $x^{2\alpha+1}\, dx$ on the interval $(0,1)$.

Other variants of Bessel orthogonal systems exist in the literature, 
see~\cite{BP1, BP2, Wng}. For instance, one can take the functions 
$2^{1/2} |J_{\alpha+1}(s_n)|^{-1} J_\alpha(s_n x)$, which are orthonormal 
with respect to the measure $x\, dx$ on the interval $(0,1)$. 
The conditions for the boundedness of these Fourier-Bessel series, 
as can be seen in~\cite{BP2}, correspond to taking $A=B=0$ and 
$b=\alpha - \frac{2\alpha+1}{p}$ in our corollary. Another usual case 
is to take the functions $(2x)^{1/2} |J_{\alpha+1}(s_n)|^{-1} J_\alpha(s_n x)$, 
which are orthonormal with respect to the measure $dx$ on $(0,1)$. Passing 
from one orthogonality to another consists basically in changing the weights.
Then, from the weighted $L^p$ boundedness of any of these systems we 
easily deduce a corresponding weighted $L^p$ boundedness for any of the
other systems.

In the case of the Fourier-Dunkl series on $(-1,1)$ we feel, however, that the 
natural setting is to start from $J_\alpha(z) z^{-\alpha}$, since these functions, 
defined by~\eqref{serie}, are holomorphic on~$\C$; in particular, they are well 
defined on the interval~$(-1,1)$.
\end{remark}
\section{Auxiliary results}

We will need to control some basic operator in weighted $L^p$ spaces on~$(-1,1)$. 
For a function $g : (0,2) \rightarrow \R$, the Calder\'{o}n operator is defined by
\[
   Ag(x) = \frac{1}{x} \int_0^x |g(y)| \, dy
   + \int_x^2 \frac{|g(y)|}{y} \, dy,
\]
that is, the sum of the Hardy operator and its adjoint. The weighted norm inequality
\[
   \| Ag \|_{L^p((0,2),u)} \leq C \| g \|_{L^p((0,2),v)}
\]
holds for every $g \in L^p((0,2),v)$, provided that $(u,v) \in A_p^\delta(0,2)$ 
for some $\delta > 1$, and $\delta=1$ is enough if $u=v$ (see~\cite{Muck, Neuge}). 
Let us consider now the operator $J$ defined by
\[
   Jf(x) = \int_{-1}^1 \frac{f(y)}{2-x-y} \, dy
\]
for $x \in (-1,1)$ and suitable functions~$f$. With the notation $f_1(t) = f(1-t)$, 
we have
\[
   |Jf(x)| = \left| \int_0^2 \frac{f(1-t)}{1 - x + t} \, dt \right|
   \leq A(f_1)(1-x)
\]
and a simple change of variables proves that the weighted norm inequality
\[
   \| Jf \|_{L^p((-1,1),u)} \leq C \| f \|_{L^p((-1,1),v)}
\]
holds for every $f \in L^p((-1,1),v)$, provided that 
$(u,v) \in A_p^\delta(-1,1)$ for some $\delta > 1$ (or $\delta=1$ if $u=v$).

The Hilbert transform on the interval $(-1,1)$ is defined as
\[
   Hg (x) = \int_{-1}^1 \frac{g(y)}{x-y} \, dy.
\]
The above weighted norm inequality holds also for the Hilbert transform with 
the same $A_p^\delta(-1,1)$ condition (see~\cite{HMW, Neuge}). In both cases, 
the norm inequalities hold with a constant $C$ depending only on the 
$A_p^\delta$ constant of the pair $(u,v)$.

Our first objective is to obtain a suitable estimate for the kernel $K_n(x,y)$. 
With this aim, we will use some well-known properties of Bessel
(and related) functions, that can be found on~\cite{Wat}.
For the Bessel functions we have the asymptotics
\begin{equation}
\label{cotacero}
   J_\nu(z) = \frac{z^\nu}{2^\nu \Gamma(\nu+1)} + O(z^{\nu+2}),
\end{equation}
if $|z|<1$, $|\arg(z)|\leq\pi$; and
\begin{equation}
\label{cotainfinito}
   J_\nu(z) = \sqrt{\frac{2}{\pi z}} \left[ \cos\left(z-\frac{\nu\pi}2
     - \frac\pi4 \right) + O(e^{\Imag(z)}z^{-1}) \right], 
\end{equation}
if $|z| \geq 1$, $|\arg(z)| \leq \pi-\theta$. The Hankel function of the first 
kind, denoted by $H_{\nu}^{(1)}$, is defined as
\[
   H_{\nu}^{(1)}(z) = J_{\nu}(z)+iY_{\nu}(z),
\]
where $Y_{\nu}$ denotes the Weber function, given by
\begin{align*}
   Y_{\nu}(z) &= \frac{J_{\nu}(z)\cos\nu\pi - J_{-\nu}(z)}{\sin\nu\pi},
   \text{ if } \nu\notin \Z, 
   \\ 
   Y_n(z) 
   &= \lim_{\nu\to n} \frac{J_{\nu}(z)\cos\nu\pi - J_{-\nu}(z)}{\sin\nu\pi},
   \text{ if } n \in \Z.
\end{align*}
From these definitions, we have
\begin{align*}
   H_{\nu}^{(1)}(z) &= \frac{J_{-\nu}(z)-e^{-\nu \pi i}J_{\nu}(z)}{i\sin\nu\pi},
   \text{ if } \nu\notin \Z,
   \\
   H_n^{(1)}(z) 
   &= \lim_{\nu\to n}\frac{J_{-\nu}(z)-e^{-\nu\pi i} J_{\nu}(z)}{i\sin\nu\pi},
   \text{ if } n \in \Z.
\end{align*}
For the function $H_\nu^{(1)}$, the asymptotic
\begin{equation}
\label{cotainfinitoH} 
   H_{\nu}^{(1)}(z) = \sqrt{\frac{2}{\pi z}} e^{i(z-\nu\pi/2-\pi/4)}
   [C+O(z^{-1})]
\end{equation}
holds for $|z|>1$, $-\pi < \arg(z) < 2\pi$, with some constant~$C$.

As usual for the $L^p$ convergence of orthogonal expansions, the results are 
consequences of suitable estimates for the kernel $K_n(x,y)$. The next lemma 
contains an estimate for the difference between the kernel $K_n(x,y)$ and an 
integral containing the product of two $E_\alpha$ functions. This integral can 
be evaluated using Lemma~1 in \cite{CiVa}. Next, to obtain the estimate we consider an 
appropriate function in the complex plane having poles in the points $s_j$ and 
integrate this function along a suitable path.

\begin{lema}
Let $\alpha>-1$. Then, there exists some constant $C > 0$ 
such that for each $n\geq 1$ and $x, y \in (-1,1)$,
\[
   \left| K_{n}(x,y)
   - \int_{-M_n}^{M_n} E_{\alpha}(izx)  \overline{E_{\alpha}(izy)}
   \, d\mu_{\alpha}(z)
   \right|
   \leq C \left(\frac{|xy|^{-(\alpha+1/2)}}{2-x-y}+1\right),
\]
where $M_n=(s_n+s_{n+1})/2$.
\end{lema}

\begin{proof}
Using elementary algebraic manipulations, the kernel $K_n(x,y)$ can be written as
\begin{equation}
\label{kernel}
   K_n(x,y) = 2^{\alpha+1} \Gamma(\alpha+2)
   + \frac{2^{\alpha+1}\Gamma(\alpha+1)}{(xy)^{\alpha}}
   \sum_{j=1}^n
   \frac{J_{\alpha}(s_j x) J_{\alpha}(s_j y)
   + J_{\alpha+1}(s_j x)J_{\alpha+1}(s_j y)}{J_{\alpha}(s_j)^2}.
\end{equation}
Let us find a function whose residues at the points $s_j$ are the terms in the series, 
so that this series can be expressed as an integral. The identities
\[
   -J_{\alpha+1}'(z) H^{(1)}_{\alpha+1}(z) 
   + J_{\alpha+1}(z) (H^{(1)}_{\alpha+1})'(z)
   = \frac{2i}{\pi z}
\]
(see \cite[p.~76]{Whit-Wat}), and
\[
   z J_{\alpha+1}'(z) + (\alpha+1) J_{\alpha+1}(z) = - z J_{\alpha}(z),
\]
give
\[
   -J_{\alpha+1}'(s_j) H^{(1)}_{\alpha+1}(s_j) = \frac{2i}{\pi s_j}
\]
and
\[
   J_{\alpha+1}'(s_j)= - J_{\alpha}(s_j)
\]
for every $j \in \N$. Then,
\begin{align*}
   &-\frac{2i}{\pi} |xy|^{1/2}\,
   \frac{J_\alpha(s_j x)J_\alpha(s_j y) 
   + J_{\alpha+1}(s_j x)J_{\alpha+1}(s_j y)}{J_{\alpha}(s_j)^2}
   \\*
   & \qquad \qquad = -\frac{2i}{\pi} |xy|^{1/2}\,
   \frac{J_\alpha(s_j x)J_\alpha(s_j y) 
   + J_{\alpha+1}(s_j x)J_{\alpha+1}(s_j y)}{J_{\alpha+1}'(s_j)^2}
   \\
   & \qquad \qquad = |xy|^{1/2} s_jH^{(1)}_{\alpha+1}(s_j)
   \frac{J_\alpha(s_j x) J_\alpha(s_j y) 
   + J_{\alpha+1}(s_j x)J_{\alpha+1}(s_j y)}{J_{\alpha+1}'(s_j)}
   \\
   & \qquad \qquad = \lim_{z \rightarrow s_j}(z-s_j)H_{x,y}(z)
   = \Res(H_{x,y},s_j),
\end{align*}
where we define
\[
   H_{x,y}(z) = |xy|^{1/2}
   \, z H^{(1)}_{\alpha+1}(z)
   \frac{J_\alpha(z x)J_\alpha(z y) + J_{\alpha+1}(z x)J_{\alpha+1}(z y)}
   {J_{\alpha+1}(z)}
\]
(the factor $|xy|^{1/2}$ is taken for convenience).
The fact that $J_\nu(-z) = e^{\nu\pi i}J_{\nu}(z)$ gives 
$\Res(H_{x,y},s_j) = \Res(H_{x,y},-s_j)$.

Since the definition of $H^{(1)}_{\alpha+1}(z)$ differs in case $\alpha \in \Z$, for the rest of the proof we will assume that $\alpha \notin \Z$; the other case can be deduced by considering the limit.

The function $H_{x,y}(z)$ is analytic in 
$\C \setminus((-\infty,-M_n]\cup [M_n,\infty)\cup \{\pm s_j: j=1,2,\dots\})$. 
Moreover, the points $\pm s_j$ are simple poles. So, we have
\begin{equation}
\label{intprinc}
   \int_{\mathbf{S} \cup \mathbf{I}(\varepsilon)} H_{x,y}(z) \, dz = 0,
\end{equation}
where $\mathbf{I}(\varepsilon)$ is the interval $[-M_n,M_n]$ warped with upper half circles 
of radius $\varepsilon$ centered in $\pm s_j$, with $j=1,\dots,n$ and $\mathbf{S}$ 
is the path of integration given by the interval $M_n+i[0,\infty)$ in the 
direction of increasing imaginary part and the interval $-M_n+i[0,\infty)$ in 
the opposite direction. The existence of the
integral is clear for the path $\mathbf{I}(\varepsilon)$; for $\mathbf{S}$ this
fact can be checked by using \eqref{cotacero}, \eqref{cotainfinito} and
\eqref{cotainfinitoH}. Indeed, on $\mathbf{S}$ we obtain that
$\Bigl|\frac{H_{\alpha+1}^{(1)}(z)}{J_{\alpha+1}(z)}\Bigr| \leq C
e^{-2\Imag(z)}$. Similarly, on $\mathbf{S}$ one has
\[
   \left| |xy|^{1/2} z J_{\alpha}(z x)J_{\alpha}(z y) \right|
   \leq C e^{\Imag(z)(x+y)} h^\alpha_{x,y}(|z|)
\]
where
\[
   h^{\alpha}_{x,y}(|z|)
   = \max\{|x z|^{\alpha+1/2},1\} \max\{|y z|^{\alpha+1/2},1\}
\]
for $-1<\alpha<-1/2$, and
\[
   h^{\alpha}_{x,y}(|z|) = 1
\]
for $\alpha\geq -1/2$. Thus
\begin{equation}
\label{boundH}
   |H_{x,y}(z)|
   \leq C \left( h^{\alpha}_{x,y}(|z|) + h_{x,y}^{\alpha+1}(|z|) \right)
   e^{-\Imag(z)(2-x-y)},
\end{equation}
and the integral on $\mathbf{S}$ is well defined.

From the definition of $H_{x,y}(z)$, we have
\begin{multline*}
   \int_{\mathbf{I}(\varepsilon)} H_{x,y}(z) \, dz
   = \int_{\mathbf{I}(\varepsilon)}
   \frac{|xy|^{1/2} z J_{-\alpha-1}(z)}{i \sin (\alpha+1)\pi}
   \cdot
   \frac{J_\alpha(zx) J_\alpha(zy) + J_{\alpha+1}(zx) J_{\alpha+1}(zy)}
   {J_{\alpha+1}(z)}
   \, dz
   \\*
   - |xy|^{1/2} \frac{e^{-(\alpha+1)\pi i}}{i \sin (\alpha+1)\pi}
   \int_{\mathbf{I}(\varepsilon)} z \left(
   J_\alpha(zx) J_\alpha(zy) + J_{\alpha+1}(zx) J_{\alpha+1}(zy)
   \right) \, dz.
\end{multline*}
The function in the first integral is odd, and the function 
in the second integral has no poles at the points~$s_j$. Then, 
the first integral equals the integral over the symmetric path 
$-\mathbf{I}(\varepsilon) = \{z : -z \in \mathbf{I}(\varepsilon)\}$. 
Putting $|z-s_j| = \varepsilon$ for the positively oriented circle, this gives
\begin{align*}
   \lim_{\varepsilon \to 0} \int_{\mathbf{I}(\varepsilon)}H_{x,y}(z) \, dz
   &= \lim_{\varepsilon \to 0} \frac{-1}{2} \sum_{|s_j| < M_n}
   \int_{|z-s_j|=\varepsilon}
   \frac{|xy|^{1/2} z J_{-\alpha-1}(z)}{i \sin (\alpha+1)\pi}
   \cdot
   \frac{J_\alpha(zx) J_\alpha(zy) + J_{\alpha+1}(zx) J_{\alpha+1}(zy)}
   {J_{\alpha+1}(z)}
   \, dz
   \\*
   &\qquad - |xy|^{1/2} \frac{e^{-(\alpha+1)\pi i}}{i \sin (\alpha+1)\pi}
   \int_{-M_n}^{M_n} z \left(
   J_\alpha(zx) J_\alpha(zy) + J_{\alpha+1}(zx) J_{\alpha+1}(zy)
   \right) \, dz
   \\
   &= - \pi i \sum_{|s_j| < M_n} \Res(H_{x,y}, s_j)
   \\*
   &\qquad - |xy|^{1/2} \frac{e^{-(\alpha+1)\pi i}}{i \sin (\alpha+1)\pi}
   (1 - e^{2\pi i \alpha})
   \int_0^{M_n} z \left(
   J_\alpha(zx) J_\alpha(zy) + J_{\alpha+1}(zx) J_{\alpha+1}(zy)
   \right) \, dz
   \\
   &= -4|xy|^{1/2} \sum_{j=1}^n \frac{J_{\alpha}(s_j x)J_{\alpha}(s_j y)
   + J_{\alpha+1}(s_j x)J_{\alpha+1}(s_j y)}{J_{\alpha}(s_j)^2}
   \\*
   &\qquad + 2 |xy|^{1/2} \int_0^{M_n} z \left(
   J_\alpha(zx) J_\alpha(zy) + J_{\alpha+1}(zx) J_{\alpha+1}(zy)
   \right) \, dz.
\end{align*}
This, together with~\eqref{intprinc}, gives
\begin{multline*}
   \quad
   \sum_{j=1}^n \frac{J_{\alpha}(s_j x)J_{\alpha}(s_j y)
   + J_{\alpha+1}(s_j x)J_{\alpha+1}(s_j y)}{J_{\alpha}(s_j)^2}
   \\
   = \frac{1}{4 |xy|^{1/2}} \int_{\mathbf{S}} H_{x,y}(z) \, dz
   + \frac{1}{2} \int_0^{M_n} z \left(
   J_\alpha(zx) J_\alpha(zy) + J_{\alpha+1}(zx) J_{\alpha+1}(zy)
   \right) \, dz.
   \quad
\end{multline*}
Then, it follows from~\eqref{kernel} that
\begin{multline*}
   \qquad
   K_n(x,y)
   = 2^{\alpha+1}\Gamma(\alpha+2)
   + \frac{2^{\alpha-1} \Gamma(\alpha+1)}{(xy)^\alpha|xy|^{1/2}}
   \int_{\mathbf{S}} H_{x,y}(z)\,dz
   \\*
   + \frac{2^\alpha\Gamma(\alpha+1)}{(xy)^\alpha}
   \int_0^{M_n} z(J_\alpha(zx)J_\alpha(zy) + J_{\alpha+1}(z x)J_{\alpha+1}(z y)) \,dz.
   \qquad
\end{multline*}
Now, it is easy to check the identity
\[
   \frac{2^\alpha \Gamma(\alpha+1)}{(xy)^\alpha}
   \int_0^{M_n}
   z(J_\alpha(zx)J_\alpha(zy) + J_{\alpha+1}(z x)J_{\alpha+1}(zy))\, dz
   = \int_{-M_n}^{M_n}
   E_{\alpha}(izx)\overline{E_{\alpha}(izy)}\,d\mu_{\alpha}(z),
\]
so that
\[
   \left| K_{n}(x,y)
   - \int_{-M_n}^{M_n}
   E_{\alpha}(izx)\overline{E_{\alpha}(izy)}\, d\mu_{\alpha}(z) \right|
   \leq
   2^{\alpha+1}\Gamma(\alpha+2)
   + \frac{2^{\alpha-1} \Gamma(\alpha+1)}{|xy|^{\alpha+1/2}}
   \left| \int_{\mathbf{S}}H_{x,y}(z)\,dz \right|.
\]
We conclude showing that
\begin{equation}
\label{estH}
   \left|\int_{\mathbf{S}}H_{x,y}(z)\, dz\right|
   \leq C \left( \frac{1}{2-x-y} + |xy|^{\alpha+1/2} \right),
\end{equation}
for $-1<x,y<1$. For $\alpha \geq -1/2$, the bound \eqref{estH}
follows from~\eqref{boundH}. Indeed, in this case
\[
   \left| \int_{\mathbf{S}}H_{x,y}(z) \, dz\right|
   \leq C \int_0^\infty e^{-t(2-x-y)} \,dt
   = \frac{C}{2-x-y}.
\]
For $-1<\alpha<-1/2$, we have 
$|H_{x,y}(z)| \leq C |xy|^{\alpha+1/2}e^{-\Imag(z)(2-x-y)}$ if 
$z \in \mathbf{S}$. With this inequality we obtain~\eqref{estH} as follows:
\[
  \left| \int_{\mathbf{S}}H_{x,y}(z)\, dz \right|
  \leq C |xy|^{\alpha+1/2} \int_0^\infty e^{-t(2-x-y)} \,dt
  = C \frac{|xy|^{\alpha+1/2}}{2-x-y}
  \leq C \left(|xy|^{\alpha+1/2} + \frac{1}{2-x-y} \right).
  \qedhere
\]
\end{proof}

From the previous lemma and the identity (see \cite{CiVa})
\[
  \int_{-1}^1 E_{\alpha}(ixz)\overline{E_{\alpha}(iyz)}\, d\mu_{\alpha}(z)
   = \frac{1}{2^{\alpha+1}\Gamma(\alpha+2)}
   \frac{x\mathcal{I}_{\alpha+1}(ix)\mathcal{I}_{\alpha}(iy)
   - y\mathcal{I}_{\alpha+1}(iy)\mathcal{I}_{\alpha}(ix)}{x-y},
\]
which holds for $\alpha > -1$, $x,y\in \C$, and $x \not= y$, we obtain that
\begin{equation}
\label{descKn}
  \left| K_{n}(x,y)
    - B(M_n,x,y)-B(M_n,y,x) \right|
  \leq C \left(\frac{|xy|^{-(\alpha+1/2)}}{2-x-y}+1\right)
\end{equation}
with
\[
   B(M_n,x,y) = \frac{M_n^{2(\alpha+1)}}{2^{\alpha+1}\Gamma(\alpha+2)}
   \frac{x\mathcal{I}_{\alpha+1}(iM_nx) \mathcal{I}_{\alpha}(iM_ny)}{x-y}
\]
or, by the definition of $\mathcal{I}_{\alpha}$ and the fact that 
$\frac{J_\alpha(z)}{z^\alpha}$ is even,
\[
   B(M_n,x,y) = 2^\alpha \Gamma(\alpha+1) 
   \frac{M_n x J_{\alpha+1}(M_n|x|) J_\alpha(M_n|y|)}
   {|x|^{\alpha+1} |y|^\alpha (x-y)}.
\]

\section{Proof of Theorem~\ref{tma1}}

We can split the partial sum operator $S_n$ into three terms suitable 
to apply~\eqref{descKn}:
\begin{align}
   S_nf(x) &= \int_{-1}^1 f(y) B(M_n,x,y) \, d\mu_\alpha(y)
   + \int_{-1}^1 f(y) B(M_n,y,x) \, d\mu_\alpha(y) 
   \notag
   \\*
   &\qquad+ \int_{-1}^1 f(y)
   \Big[ K_{n}(x,y) - B(M_n,x,y)-B(M_n,y,x) \Big] \, d\mu_\alpha(y)
   \notag
   \\
   &=: T_{1,n}f(x) + T_{2,n}f(x) + T_{3,n}f(x).
   \label{T123}
\end{align}
With this decomposition, the theorem will be proved if we see that
\[
  \| U T_{j,n}f \|^p_{L^p((-1,1),d\mu_\alpha)} 
  \leq C \| V f \|^p_{L^p((-1,1),d\mu_\alpha)},
  \qquad j = 1,2,3,
\]
for a constant $C$ independent of $n$ and~$f$.

\subsection{The first term} 
We have
\begin{align*}
   T_{1,n}f(x) &= \frac{1}{2^{\alpha+1} \Gamma(\alpha+1)}
   \int_{-1}^1 f(y) B(M_n,x,y) |y|^{2\alpha+1} \, dy
   \\
   &= \frac{M_n^{1/2} x J_{\alpha+1}(M_n |x|)}{2 |x|^{\alpha+1}}
   \int_{-1}^1 \frac{f(y) M_n^{1/2} J_\alpha(M_n|y|) |y|^{\alpha+1}}{x-y} \, dy.
\end{align*}
According to~\eqref{cotacero} and~\eqref{cotainfinito} and the 
assumption that $\alpha \geq -1/2$, we have
\[
   |J_\alpha(z)| \leq C z^{-1/2},
   \qquad
   |J_{\alpha+1}(z)| \leq C z^{-1/2},
\]
for every $z > 0$. Using these inequalities and the boundedness of the 
Hilbert transform under the $A_p$ condition~\eqref{Ap_tma1} gives
\begin{align*}
   &\| U T_{1,n}f \|^p_{L^p((-1,1),d\mu_\alpha)}
   \\*
   &\qquad= C \int_{-1}^1 \left|
   \int_{-1}^1 \frac{f(y) M_n^{1/2} J_\alpha(M_n|y|) |y|^{\alpha+1}}{x-y} \, dy
   \right|^p
   U(x)^p M_n^{p/2} |J_{\alpha+1}(M_n |x|)|^p |x|^{2\alpha+1-\alpha p} \, dx
   \\
   &\qquad\leq C \int_{-1}^1 \left|
   \int_{-1}^1 \frac{f(y) M_n^{1/2} J_\alpha(M_n|y|) |y|^{\alpha+1}}{x-y} \, dy
   \right|^p
   U(x)^p |x|^{(\alpha+\frac{1}{2})(2- p)} \, dx
   \\
   &\qquad\leq C \int_{-1}^1 \left|
   f(x) M_n^{1/2} J_\alpha(M_n|x|) |x|^{\alpha+1}
   \right|^p
   V(x)^p |x|^{(\alpha+\frac{1}{2})(2- p)} \, dx
   \\
   &\qquad\leq C \int_{-1}^1 |f(x)|^p V(x)^p |x|^{2\alpha+1} \, dx
   = C \| V f \|^p_{L^p((-1,1),d\mu_\alpha)}.
\end{align*}

\subsection{The second term} 
This term is given by
\begin{align*}
   T_{2,n}f(x) &= \frac{1}{2^{\alpha+1} \Gamma(\alpha+1)}
   \int_{-1}^1 f(y) B(M_n,y,x) |y|^{2\alpha+1} \, dy
   \\
   &= \frac{M_n^{1/2} J_\alpha(M_n |x|)}{2 |x|^\alpha}
   \int_{-1}^1 \frac{f(y) y M_n^{1/2} J_{\alpha+1}(M_n|y|) |y|^\alpha}{y-x}
   \, dy
\end{align*}
and everything goes as with the first term.

\subsection{The third term} 
According to~\eqref{descKn},
\[
   |T_{3,n}f(x)| 
   \leq C |x|^{-(\alpha+1/2)} 
   \int_{-1}^1 \frac{f(y) |y|^{\alpha+1/2}}{2-x-y} \, dy
   + C \int_{-1}^1 |f(y)| \, |y|^{2\alpha+1} \, dy
\]
so it is enough to have both
\begin{equation}
\label{T3-1}
   \int_{-1}^1 
   \left| \int_{-1}^1 \frac{f(y) |y|^{\alpha+1/2}}{2-x-y} \, dy \right|^p
   U(x)^p |x|^{2\alpha+1-p(\alpha+1/2)} \, dx
\end{equation}
and
\begin{equation}
\label{T3-2}
   \left| \int_{-1}^1 |f(x)| \, |x|^{2\alpha+1} \, dx \right|^p
   \int_{-1}^1 U(x)^p |x|^{2\alpha+1} \, dx
\end{equation}
bounded by
\[
   C \int_{-1}^1 |f(x)|^p V(x)^p |x|^{2\alpha+1} \, dx.
\]
For the boundedness of~\eqref{T3-1} it suffices to impose
\[
   \Big(
   U(x)^p |x|^{2\alpha+1-p(\alpha+1/2)},
   V(x)^p |x|^{2\alpha+1-p(\alpha+1/2)}
   \Big) \in A_p^\delta(-1,1),
\]
but this is exactly~\eqref{Ap_tma1}. By duality, the boundedness of~\eqref{T3-2} 
is equivalent to
\[
   \left( \int_{-1}^1 U(x)^p |x|^{2\alpha+1} \, dx \right)
   \left(
   \int_{-1}^1 V(x)^{-p/(p-1)} |x|^{2\alpha+1} \, dx
   \right)^{p-1} < \infty.
\]
Now, it is easy to check that
\begin{multline*}
   \left( \int_{-1}^1 U(x)^p |x|^{2\alpha+1} \, dx \right)
   \left(
   \int_{-1}^1 V(x)^{-p/(p-1)} |x|^{2\alpha+1} \, dx
   \right)^{p-1}
   \\
   \leq
   \left( \int_{-1}^1 U(x)^p |x|^{(\alpha+\frac{1}{2})(2-p)} \, dx \right)
   \left(
   \int_{-1}^1 \left(
   V(x)^p |x|^{(\alpha+\frac{1}{2})(2-p)}
   \right)^{-\frac{1}{p-1}} \, dx
   \right)^{p-1}
   \leq C,
\end{multline*}
the last inequality following from the $A_p$ condition~\eqref{Ap_tma1}.

\section{Proof of Theorem~\ref{tma2}}

We begin with a simple lemma on $A_p$ weights.

\begin{lema}
\label{pesosAp}
Let $1 < p < \infty$,  $(u,v) \in A_p(-1,1)$, $(u_1,v_1) \in A_p(-1,1)$. 
Let $w$, $\zeta$ be weights on $(-1,1)$ such that either
\[
   w \leq C (u + u_1)
   \quad\text{and}\quad
   \zeta \geq C_1(v + v_1)
\]
or
\[
   w^{-1} \geq C (u^{-1} + u_1^{-1})
   \quad\text{and}\quad
   \zeta^{-1} \leq C_1(v^{-1} + v_1^{-1})
\]
for some constants $C$, $C_1$. Then $(w,\zeta) \in A_p(-1,1)$ with a constant 
depending only on $C$, $C_1$ and the $A_p$ constants of $(u,v)$ and $(u_1, v_1)$.
\end{lema}

\begin{proof} Assume that $w \leq C (u + u_1)$ and $\zeta \geq C_1(v + v_1)$.
For any interval $I \subseteq (-1,1)$, 
\[
   \left(\frac{1}{|I|} \int_I \zeta^{-\frac{1}{p-1}} \right)^{p-1}
   \leq
   \frac{1}{C_1} \min\left\{
   \left(\frac{1}{|I|} \int_I v^{-\frac{1}{p-1}} \right)^{p-1},
   \left(\frac{1}{|I|} \int_I v_1^{-\frac{1}{p-1}} \right)^{p-1}
    \right\}.
\]
Therefore,
\[
   \left(\frac{1}{|I|} \int_I w \right)
   \left(\frac{1}{|I|} \int_I \zeta^{-\frac{1}{p-1}} \right)^{p-1}
   \leq \frac{C}{C_1} \left( \frac{1}{|I|} \int_I u \right)
   \left(\frac{1}{|I|} \int_I v^{-\frac{1}{p-1}} \right)^{p-1} +
   \frac{C}{C_1} \left( \frac{1}{|I|} \int_I u_1 \right)
   \left(\frac{1}{|I|} \int_I v_1^{-\frac{1}{p-1}} \right)^{p-1}.
\]
This proves that $(w,\zeta) \in A_p(-1,1)$ with a constant depending 
on $C$, $C_1$ and the $A_p$ constants of $(u,v)$ and $(u_1,v_1)$.

Assume now that $w^{-1} \geq C (u^{-1} + u_1^{-1})$ and 
$\zeta^{-1} \leq C_1(v^{-1} + v_1^{-1})$. Then
\begin{equation}
\label{intw}
   \frac{1}{|I|} \int_I w
   \leq
   \frac{1}{C} \min\left\{
   \frac{1}{|I|} \int_I u,
   \frac{1}{|I|} \int_I u_1
    \right\}
\end{equation}
for any interval $I \subseteq (-1,1)$. On the other hand, the inequality
\begin{equation}
\label{ablanda}
   \frac{1}{2} (a^\lambda + b^\lambda) 
   \leq (a+b)^\lambda \leq 2^\lambda (a^\lambda + b^\lambda),
   \qquad a, b \geq 0, \ \lambda > 0
\end{equation}
gives
\[
   \zeta^{-\frac{1}{p-1}}
   \leq C_1^\frac{1}{p-1} \big(v^{-1} + v_1^{-1}\big)^\frac{1}{p-1}
   \leq
   C_1^\frac{1}{p-1} 2^\frac{1}{p-1} \big(v^{-\frac{1}{p-1}}
   + v_1^{-\frac{1}{p-1}}\big),
\]
and
\[
   \left( \frac{1}{|I|} \int_I \zeta^{-\frac{1}{p-1}} \right)^{p-1}
   \leq
   2^p C_1 \left(
   \frac{1}{|I|} \int_I v^{-\frac{1}{p-1}} \right)^{p-1}
   + 2^p C_1 \left( \frac{1}{|I|} \int_I v_1^{-\frac{1}{p-1}}
   \right)^{p-1}.
\]
This, together with~\eqref{intw}, proves that $(w,\zeta) \in A_p(-1,1)$ with a 
constant depending on $C$, $C_1$ and the $A_p$ constants of $(u,v)$ and $(u_1,v_1)$.
\end{proof}

Now, we use the following estimate for the Bessel functions, which is a consequence 
of~\eqref{cotacero}, \eqref{cotainfinito} and $-1 < \alpha < -1/2$:
\[
   |z^{1/2} J_{\alpha}(z)|
   \leq C (1 + z^{\alpha+1/2}),
   \quad z \geq 0,
\]
and
\[
   |z^{1/2} J_{\alpha+1}(z)|
   \leq C {(1 + z^{\alpha+1/2})}^{-1},
   \quad z \geq 0.
\]
In particular, there exists a constant $C$ such that, for $x \in (-1,1)$ and $n \geq 0$, 
we have
\[
   M_n^{1/2} |J_\alpha(M_n|x|)|
   \leq C |x|^{-1/2} (1 + {|M_nx|}^{\alpha+1/2})
\]
and
\[
   M_n^{1/2} |J_{\alpha+1}(M_n|x|)| 
   \leq C\, \frac{|x|^{-1/2}}{1 + {|M_nx|}^{\alpha+1/2}}.
\]
Moreover, the inequality~\eqref{ablanda} gives
\[
   2^{\alpha+1/2}|x|^{\alpha+1/2} (|x| + M_n^{-1})^{-(\alpha+1/2)}
   \leq 1 + |M_nx|^{\alpha+1/2}
   \leq 2 |x|^{\alpha+1/2} (|x| + M_n^{-1})^{-(\alpha+1/2)}
\]
so that we get
\begin{equation}
\label{cotaJaMn}
   M_n^{1/2} |J_\alpha(M_n|x|)|
   \leq C |x|^\alpha (|x| + M_n^{-1})^{-(\alpha+1/2)}
\end{equation}
and
\begin{equation}
\label{cotaJa1Mn}
   M_n^{1/2} |J_{\alpha+1}(M_n|x|)| \leq C |x|^{-(\alpha+1)} (|x| + M_n^{-1})^{\alpha+1/2}.
\end{equation}

To handle these expressions, the following result will be useful:

\begin{lema}
\label{Apunif}
Let $1 < p < \infty$, a sequence $\{M_n\}$ of positive numbers that tends to infinity, 
two nonnegative functions $U$ and $V$ defined on the interval $(-1,1)$, 
$-1 < \alpha < -1/2$ and $\delta > 1$ ($\delta = 1$ if $U=V$). If~\eqref{Ap1_tma2} 
and~\eqref{Ap2_tma2} are satisfied, then
\begin{equation}
\label{Ap1unif}
   \Big(U(x)^p {(|x|+M_n^{-1})}^{p(\alpha+1/2)} |x|^{(2\alpha+1)(1-p)},
   V(x)^p {(|x|+M_n^{-1})}^{p(\alpha+1/2)} |x|^{(2\alpha+1)(1-p)} \Big) \in A_p^\delta(-1,1),
\end{equation}
\begin{equation}
\label{Ap2unif}
   \Big(U(x)^p {(|x|+M_n^{-1})}^{-p(\alpha+1/2)} |x|^{2\alpha+1},
   V(x)^p {(|x|+M_n^{-1})}^{-p(\alpha+1/2)} |x|^{2\alpha+1} \Big) \in A_p^\delta(-1,1),
\end{equation}
``uniformly'', i.e., with $A_p^\delta$ constants independent of~$n$.
\end{lema}

\begin{proof}
As a first step, let us observe that~\eqref{Ap1_tma2} and~\eqref{Ap2_tma2} imply
\[
   \left(
   U(x)^p |x|^{(2\alpha+1)(1-\frac{1}{2}p)},
   V(x)^p |x|^{(2\alpha+1)(1-\frac{1}{2}p)}
   \right)
   \in A_p^\delta(-1,1).
\]
To prove this, just put
\[
   U(x)^p |x|^{(2\alpha+1)(1-\frac{1}{2}p)}
   = \left[ U(x)^p |x|^{(2\alpha+1)(1-p)} \right]^{1/2}
   \left[ U(x)^p |x|^{(2\alpha+1)} \right]^{1/2}
\]
(the same with $V$) and check the $A_p^\delta$ condition using the Cauchy-Schwarz 
inequality and~\eqref{Ap1_tma2},~\eqref{Ap2_tma2}.

Now, \eqref{ablanda} yields
\begin{multline*}
   \qquad
   \left[
   U(x)^p {(|x|+M_n^{-1})}^{p(\alpha+\frac{1}{2})} |x|^{(2\alpha+1)(1-p)}
   \right]^{-\delta}
   \\
   \geq 
   \frac{1}{2} \left[
   U(x)^p |x|^{(2\alpha+1)(1-\frac{1}{2} p)}
   \right]^{-\delta}
   + \frac{1}{2} \left[
   U(x)^p M_n^{-p(\alpha+\frac{1}{2})} |x|^{(2\alpha+1)(1-p)}
   \right]^{-\delta}
   \qquad
\end{multline*}
and
\begin{multline*}
   \left[
   V(x)^p {(|x|+M_n^{-1})}^{p(\alpha+\frac{1}{2})} |x|^{(2\alpha+1)(1-p)}
   \right]^{-\delta}
   \\
   \leq 2^{-p\delta(\alpha+\frac{1}{2})} \left[
   V(x)^p |x|^{(2\alpha+1)(1-\frac{1}{2}p)}
   \right]^{-\delta}
   + 2^{-p\delta(\alpha+\frac{1}{2})} \left[
   V(x)^p M_n^{-p(\alpha+\frac{1}{2})} |x|^{(2\alpha+1)(1-p)}
   \right]^{-\delta}.
\end{multline*}
Thus, Lemma~\ref{pesosAp} gives~\eqref{Ap1unif} with an $A_p^\delta$ 
constant independent of~$n$, since the $A_p^\delta$ constant of the pair
\[
   \left(
   U(x)^p M_n^{-p(\alpha+\frac{1}{2})} |x|^{(2\alpha+1)(1-p)},
   V(x)^p M_n^{-p(\alpha+\frac{1}{2})} |x|^{(2\alpha+1)(1-p)}
   \right)
\]
is the same constant of the pair
\[
   \left(
   U(x)^p |x|^{(2\alpha+1)(1-p)},
   V(x)^p |x|^{(2\alpha+1)(1-p)}
   \right)
\]
i.e., it does not depend on $n$. The proof of~\eqref{Ap2unif} follows the same argument, since
\begin{multline*}
   \qquad
   \left[
   U(x)^p {(|x|+M_n^{-1})}^{-p(\alpha+\frac{1}{2})} |x|^{2\alpha+1}
   \right]^\delta
   \\
   \leq 
   2^{-p\delta(\alpha+\frac{1}{2})}
   \left[ U(x)^p |x|^{(2\alpha+1)(1-\frac{1}{2} p)} \right]^\delta
   + 2^{-p\delta(\alpha+\frac{1}{2})}
   \left[ U(x)^p M_n^{p(\alpha+\frac{1}{2})} |x|^{2\alpha+1} \right]^\delta
   \qquad
\end{multline*}
and
\[
   \left[
   V(x)^p {(|x|+M_n^{-1})}^{-p(\alpha+\frac{1}{2})} |x|^{2\alpha+1}
   \right]^\delta
   \geq 
   \frac{1}{2}
   \left[ V(x)^p |x|^{(2\alpha+1)(1-\frac{1}{2} p)} \right]^\delta
   + \frac{1}{2}
   \left[ V(x)^p M_n^{p(\alpha+\frac{1}{2})} |x|^{2\alpha+1} \right]^\delta.
   \qedhere
\]
\end{proof}

We already have all the ingredients to start with the proof of Theorem~\ref{tma2}.
Let us take the same decomposition $S_n f = T_{1,n}f + T_{2,n} + T_{3,n}f$ 
as in~\eqref{T123} in the previous section and consider each term separately. 

\subsection{The first term} 
As in the proof of Theorem~\ref{tma1}, by using~\eqref{cotaJa1Mn} we have
\begin{gather*}
   \| U T_{1,n}f \|^p_{L^p((-1,1),d\mu_\alpha)} 
   = \int_{-1}^1 \left|
   \int_{-1}^1 \frac{f(y) M_n^{1/2} J_\alpha(M_n|y|) |y|^{\alpha+1}}{x-y} \, dy
   \right|^p
   U(x)^p M_n^{p/2} |J_{\alpha+1}(M_n|x|)|^p |x|^{2\alpha+1-\alpha p} \, dx
   \\
   \leq C \int_{-1}^1 \left|
   \int_{-1}^1 \frac{f(y) M_n^{1/2} J_\alpha(M_n|y|) |y|^{\alpha+1}}{x-y} \, dy
   \right|^p
   U(x)^p (|x| + M_n^{-1})^{p(\alpha+1/2)} |x|^{(2\alpha+1)(1-p)} \, dx.
\end{gather*}
Now, by the $A_p$ condition~\eqref{Ap1unif}, this is bounded by
\[
   C \int_{-1}^1 \left|
   f(x) M_n^{1/2} J_\alpha(M_n|x|) |x|^{\alpha+1}
   \right|^p
   V(x)^p (|x| + M_n^{-1})^{p(\alpha+1/2)} |x|^{(2\alpha+1)(1-p)} \, dx,
\]
which, by~\eqref{cotaJaMn} is in turn bounded by
\[
   C \int_{-1}^1 |f(x)|^p V(x)^p |x|^{2\alpha + 1} \, dx
   = C \| V f \|^p_{L^p((-1,1),d\mu_\alpha)}.
\]

\subsection{The second term}
The definition of $T_{2,n}$ and~\eqref{cotaJaMn} yield
\begin{gather*}
   \| U T_{2,n}f \|^p_{L^p((-1,1),d\mu_\alpha)}
   = \int_{-1}^1 \left|
   \int_{-1}^1
   \frac{f(y) y M_n^{1/2} J_{\alpha+1}(M_n|y|) |y|^\alpha}{y-x} \, dy
   \right|^p
   U(x)^p M_n^{p/2} |J_{\alpha}(M_n|x|)|^p |x|^{2\alpha+1-\alpha p} \, dx
   \\
   \leq C \int_{-1}^1 \left|
   \int_{-1}^1
   \frac{f(y) y M_n^{1/2} J_{\alpha+1}(M_n|y|) |y|^\alpha}{y-x} \, dy
   \right|^p
   U(x)^p (|x| + M_n^{-1})^{-p(\alpha+1/2)} |x|^{2\alpha+1} \, dx.
\end{gather*}
Now, by the $A_p$ condition~\eqref{Ap2unif}, this is bounded by
\[
   C \int_{-1}^1 \left|
   f(x) x M_n^{1/2} J_{\alpha+1}(M_n|x|) |x|^\alpha
   \right|^p
   V(x)^p (|x| + M_n^{-1})^{-p(\alpha+1/2)} |x|^{2\alpha+1} \, dx,
\]
which, by~\eqref{cotaJa1Mn} is in turn bounded by
\[
   C \int_{-1}^1 |f(x)|^p V(x)^p |x|^{2\alpha + 1} \, dx
   = C \| V f \|^p_{L^p((-1,1),d\mu_\alpha)}.
\]

\subsection{The third term} 
Taking limits when $n \to \infty$ in~\eqref{Ap1unif} we get~\eqref{Ap_tma1}, 
so the proof of the boundedness of the third summand in Theorem~\ref{tma1} 
is still valid for Theorem~\ref{tma2}.

\section{Proof of Theorem~\ref{tma3}}

The following lemma is a small variant of a result proved in~\cite{GPRV}. 
We give here a proof for the sake of completeness.

\begin{lema}
\label{lemalimite}
Let $\nu>-1$. Let $h$ be a Lebesgue measurable nonnegative function on $[0,1]$, 
$\{\rho_n\}$ a positive sequence such that
$\displaystyle\lim_{n \to \infty} \rho_n = +\infty$ and $1 \leq p < \infty$. Then
\begin{equation}
\label{lemaFejer}
   \lim_{n \to \infty} \int_0^1 |\rho_n^{1/2} J_\nu(\rho_n x)|^p h(x) \, dx
   \geq M \int_0^1 h(x) x^{-p/2} \, dx
\end{equation}
(in particular, that limit exists), where $M$ is a positive constant independent of $h$ 
and~$\{\rho_n\}$.
\end{lema}

\begin{proof}
We can assume that $h(x) x^{\nu p}$ is integrable on $(0,\delta)$ for some $\delta \in (0,1)$, 
since otherwise
\[
   \int_0^1 |\rho_n^{1/2} J_\nu(\rho_n x)|^p h(x) \, dx = \infty
\]
for each $n$, as follows from~\eqref{cotacero}, and~\eqref{lemaFejer} is trivial. 
Assume also for the moment that $h(x) x^{-p/2}$ is integrable on $(0,1)$. 
For each $x \in (0,1)$ and $n$, let us put
\[
   \varphi(x,n) = (\rho_n x)^{1/2} J_\nu(\rho_n x) 
   - \sqrt{\frac{2}{\pi}} 
   \cos\left(\rho_n x - \frac{\nu \pi}{2} - \frac{\pi}{4}\right).
\]
The estimate~\eqref{cotainfinito} gives
\[
   \lim_{n \to \infty} \varphi(x,n) = 0
\]
for each $x \in (0,1)$. Moreover, in case $\rho_n x \geq 1$ the same estimate gives
\begin{equation}
\label{fxn1}
   |\varphi(x,n)| \leq \frac{C}{\rho_n x} \leq C
\end{equation}
with a constant $C$ independent of $n$ and $x$, while for $\rho_n x \leq 1$ it 
follows from~\eqref{cotacero} that
\begin{equation}
\label{fxn2}
   |\varphi(x,n)| \leq C \left( (\rho_n x)^{\nu+{1}/{2}} + 1 \right).
\end{equation}
Without loss of generality we can assume that $\rho_n \geq 1$. Then, \eqref{fxn1} and~\eqref{fxn2} 
give $|\varphi(x,n)| \leq C (x^{\nu+{1}/{2}}+1)$ with a constant $C$ independent of $x$ and~$n$, 
so that, by the dominate convergence theorem,
\begin{equation}
\label{primerpaso}
   \lim_{n \to \infty} \int_0^1 \left|
   (\rho_n x)^{1/2} J_\nu(\rho_n x) 
   - \sqrt{\frac{2}{\pi}} 
   \cos\left(\rho_n x - \frac{\nu \pi}{2} - \frac{\pi}{4}\right)
   \right|^p h(x) x^{-p/2} \, dx = 0.
\end{equation}
Therefore,
\begin{equation}
\label{paraFejer}
   \lim_{n \to \infty} \int_0^1 |\rho_n^{1/2} J_\nu(\rho_n x)|^p h(x) \, dx
   = \lim_{n \to \infty} \int_0^1 \left| \sqrt{\frac{2}{\pi}} 
   \cos\left(\rho_n x - \frac{\nu \pi}{2} - \frac{\pi}{4}\right)
   \right|^p h(x) x^{-p/2} \, dx.
\end{equation}
Now we use Fej\'er's lemma: if $f \in L^1(0,2\pi)$, and $g$ is a continuous, 
$2\pi$-periodic function, then
\[
   \lim_{\lambda \to \infty}
   \frac{1}{2\pi} \int_0^{2\pi} g(\lambda t) f(t) \, dt
   = \widehat{g}(0) \widehat{f}(0)
   = \frac{1}{2\pi} \int_0^\pi g(t) \, dt
   \;
   \frac{1}{2\pi} \int_0^\pi f(t) \, dt
\]
where $\widehat{f}$, $\widehat{g}$ denote the Fourier transforms of $f$, $g$. 
After a change of variables, Fej\'er's lemma applied to the right hand side 
of~\eqref{paraFejer} gives
\[
   \lim_{n \to \infty} \int_0^1 |\rho_n^{1/2} J_\nu(\rho_n x)|^p h(x) \, dx
   = M \int_0^1 h(x) x^{-p/2} \, dx
\]
for some constant $M$, thus proving~\eqref{lemaFejer}.

Finally, in case $h(x) x^{-p/2}$ is not integrable on $(0,1)$, let us take the 
sequence of increasing measurable sets
\[
   K_j = \{x \in (0,1): h(x) x^{-p/2} \leq j\},
   \qquad j \in \N,
\]
and define $h_j = h$ on $K_j$ and $h_j = 0$ on $(0,1) \setminus K_j$. 
Applying~\eqref{lemaFejer} to each $h_j$ and then the monotone convergence 
theorem proves that
\[
   \lim_{n \to \infty} \int_0^1 |\rho_n^{1/2} J_\nu(\rho_n x)|^p h(x) \, dx
   = \infty,
\]
which is~\eqref{lemaFejer}.
\end{proof}

We can now prove Theorem~\ref{tma3}.

\begin{proof}[Proof of Theorem~\ref{tma3}]
The first partial sum of the Fourier expansion is
\[
   S_0 f = e_0 \int_{-1}^1 f \overline{e_0} \, d\mu_\alpha
   = (\alpha+1) \int_{-1}^1 f(x) |x|^{2\alpha+1} \, dx,
\]
so that the inequality 
$\|S_0(f) U\|_{L^p((-1,1),d\mu_{\alpha})} \leq C \|f V\|_{L^p((-1,1), d\mu_{\alpha})}$ 
gives, by duality,
\[
   U(x)^p |x|^{2\alpha+1} \in L^1((-1,1),dx),
   \quad
   V(x)^{-p'} |x|^{2\alpha+1} \in L^1((-1,1),dx).
\]
In fact, this is needed just to ensure that the partial sums of the Fourier expansions 
of all functions in $L^p(V^p \, d\mu_\alpha)$ are well defined and belong to 
$L^p(U^p \, d\mu_\alpha)$. These are the last two integrability conditions 
of Theorem~\ref{tma3}.

Now, if
\[
   \|S_n(f) U\|_{L^p((-1,1),d\mu_{\alpha})}
   \leq C \|f V\|_{L^p((-1,1), d\mu_{\alpha})}
\]
then the difference
\begin{align*}
   S_n f - S_{n-1}f
   &= e_n \int_{-1}^1 f \overline{e_n} \, d\mu_\alpha
   + e_{-n} \int_{-1}^1 f \overline{e_{-n}} \, d\mu_\alpha
   \\
   &= e_n \int_{-1}^1 f \overline{e_n} \, d\mu_\alpha
   + \overline{e_n} \int_{-1}^1 f e_n \, d\mu_\alpha
\end{align*}
is bounded in the same way. Taking even and odd functions, and using 
that $\Real e_n$ is even and $\Imag e_n$ is odd, gives
\begin{equation}
\label{cotaWing}
   \| U \Real e_n\|_{L^p((-1,1),d\mu_{\alpha})}
   \| V^{-1} \Real e_n\|_{L^{p'}((-1,1),d\mu_{\alpha})}
   \leq C
\end{equation}
and the same inequality with $\Imag e_n$. Recall that
\[
   \Real e_n(x)
   = 2^{\alpha/2} \Gamma(\alpha+1)^{1/2}\,
   \frac{|s_n|^{\alpha}}{|J_{\alpha}(s_n)|}
   \, \frac{J_{\alpha}(s_n x)}{(s_n x )^{\alpha}}.
\]
Taking into account that $|J_\nu(x)|$ is an even function (recall that 
$J_\alpha(z) / z^\alpha$ is taken as an even function) and 
$|J_\alpha(s_n)| \leq C s_n^{-1/2}$ (this follows from~\eqref{cotainfinito}),
 Lemma~\ref{lemalimite} gives
\[
   \liminf_{n \to \infty}
   \int_{-1}^1 \left|\frac{1}{J_\alpha(s_n)} J_\nu(s_n x) \right|^p h(x) \, dx
   \geq C \int_{-1}^1 h(x) |x|^{-p/2} \, dx
\]
for every measurable nonnegative function $h$. Therefore,
\[
   \liminf_{n \to \infty} \| U \Real e_n \|_{L^p((-1,1), d\mu_\alpha)}
   \geq C \left(
   \int_{-1}^1 U(x)^p |x|^{-p\alpha - \frac{p}{2} + 2\alpha + 1} \, dx
   \right)^{\frac{1}{p}}
\]
and the corresponding lower bound for 
$\liminf_n \| V^{-1} \Real e_n \|_{L^{p'}((-1,1), d\mu_\alpha)}$ holds. 
The same bounds hold for $\Imag e_n$. Thus,~\eqref{cotaWing} implies
\[
   \left(
   \int_{-1}^1 U(x)^p |x|^{-p\alpha - \frac{p}{2} + 2\alpha + 1} \, dx
   \right)^{\frac{1}{p}}
   \left(
   \int_{-1}^1 V(x)^{-p'} |x|^{-p'\alpha - \frac{p'}{2} + 2\alpha + 1} \, dx
   \right)^{\frac{1}{p'}}
   \leq C
\]
or, in other words, the first two integrability conditions of Theorem~\ref{tma3}.

Take now $f = U/(1 + V + UV)$ and any measurable set $E \subseteq (-1,1)$. 
Then $f \in L^2(d\mu_\alpha)$ by H\"older's inequality, the obvious inequality 
$|f| \leq UV^{-1}$ and the integrability conditions $U \in L^p(d\mu_\alpha)$, 
$V^{-1} \in L^{p'}(d\mu_\alpha)$, already proved. Since $\{e_j\}_{j\in\Z}$ is a 
complete orthonormal system in $L^2((-1, 1) , d\mu_{\alpha})$, we have 
$S_n(f \chi_E) \to f \chi_E$ in the $L^2(d\mu_\alpha)$ norm. Therefore, 
there exists some subsequence $S_{n_j}(f \chi_E)$ converging to $f \chi_E$ 
almost everywhere. Fatou's lemma then gives
\[
   \int_{-1}^1 | f \chi_E |^p U^p \, d\mu_\alpha
   \leq \liminf_{j \to \infty}
   \int_{-1}^1 |S_{n_j}(f \chi_E)|^p U^p \, d\mu_\alpha.
\]
Under the hypothesis of Theorem~\ref{tma3}, each of the integrals on the right 
hand side is bounded by
\[
   C^p \int_{-1}^1 |f \chi_E|^p V^p \, d\mu_\alpha
\]
(observe, by the way, that $f V \in L^p(d\mu_\alpha)$, since $|fV| \leq 1$). Thus,
\[
   \int_{-1}^1 | f \chi_E |^p U^p \, d\mu_\alpha
   \leq C^p \int_{-1}^1 |f \chi_E|^p V^p \, d\mu_\alpha
\]
for every measurable set $E \subseteq (-1,1)$. This gives $f U \leq C f V$ 
almost everywhere, and $U \leq C V$.
\end{proof}

\section*{Acknowledgment}
We thank the referee for his valuable suggestions, which helped us to make the paper more readable.


\end{document}